\documentclass{amsart}[12pt]
\usepackage{xypic}
\usepackage{eucal}

\usepackage{amsmath}
\usepackage{amstext}
\usepackage{amssymb}
\usepackage{amsthm}
\usepackage{amscd}

\usepackage{lscape}
\usepackage{longtable}

\usepackage{epsfig}
\input txdtools
\let\et=\etexdraw
\def\etexdraw{\drawbb\et}

\theoremstyle{plain}%default
\newtheorem{thm}{Theorem}[section]
\newtheorem{thm*}{Theorem}
\newtheorem{lem}[thm]{Lemma}
\newtheorem{prop}[thm]{Proposition}
\newtheorem{prop*}[thm*]{Proposition}
\newtheorem{cor}[thm]{Corollary}

\theoremstyle{definition}
\newtheorem{defi}[thm]{Definition}

\theoremstyle{remark}

\DeclareMathOperator{\Ann}{ann}
\DeclareMathOperator{\Nil}{Nil}

\begin{document}

\title
[On the discreteness and rationality of $F$-jumping coefficients]
{On the discreteness and rationality of $F$-jumping coefficients}

\author{Mordechai Katzman}
\address[Katzman]{Department of Pure Mathematics,
University of Sheffield, Hicks Building, Sheffield S3 7RH, United Kingdom}
\email{M.Katzman@sheffield.ac.uk}

\author{Gennady Lyubeznik}
\address[Lyubeznik]{Department of Mathematics, University of Minnesota, Minneapolis, MN 55455, USA}
\email{gennady@math.umn.edu}
\author{Wenliang Zhang}
\address[Zhang]{Department of Mathematics, University of Minnesota, Minneapolis, MN 55455, USA}
\email{wlzhang@math.umn.edu}
\begin{abstract}We prove that the $F$-jumping coefficients of a principal ideal of an excellent regular local ring of characteristic $p>0$ are all rational and form a discrete subset of $\mathbb R$.
\end{abstract}
\thanks{The second and third author are grateful to the NSF for support under grant DMS-0202176}

%\subjclass{Primary 13A35, 13D45, 13P99}

%\date{\today}

%\keywords{Graded commutative Noetherian ring, graded local cohomology module, infinite set of associated primes.}

%\begin{abstract}
%\end{abstract}

\maketitle
%\begin{abstract}
%\end{abstract}    F -THRESHOLDS OF HYPERSURFACES

%%%%%%%%%%%%%%%%%%%%%%%%%%%%%%%%%%%%%%%%%%%%%%%%%%%%%%%%%%%%%%%%%%%%%%%%%%%%%%%%%%%%%%%%%%%%%
\section{Introduction}\label{Section: Introduction}
%%%%%%%%%%%%%%%%%%%%%%%%%%%%%%%%%%%%%%%%%%%%%%%%%%%%%%%%%%%%%%%%%%%%%%%%%%%%%%%%%%%%%%%%%%%%%

All rings in this paper are commutative, Noetherian and of characteristic $p>0$. 

In \cite{Hara-Yoshida} Hara and Yoshida defined generalized test ideals  $\tau(\mathfrak{a}^c)\subset R$
for ideals $\mathfrak{a}$ of a ring $R$ and
non-negative parameters $c\in\mathbb{R}$.
Since then several papers studied the dependence of these ideals on $c$.
Notably, in \cite{Blickle-Mustata-Smith-1} the authors studied, under the assumption that $R$ is regular and $F$-finite, the \emph{$F$-jumping coefficients}, i.e. the
non-negative $c\in \mathbb{R}$ for which $\tau( \mathfrak{a}^{c-\epsilon} ) \neq \tau( \mathfrak{a}^{c} )$
for all $\epsilon>0$.
They showed that when $R$ is regular and essentially of finite type over an $F$-finite field,
the $F$-jumping coefficients of every ideal $\mathfrak a$ of $R$ are all rational and form a discrete subset of $\mathbb R$.

Generalized test ideals and $F$-jumping coefficients have important counterparts in characteristic 0, namely, multiplier ideals and jumping coefficients. The discreteness and rationality of $F$-jumping coefficients parallels the discreteness and rationality of jumping coefficients in characteristic 0 \cite{ELSV}.

In \cite{Blickle-Mustata-Smith-2} discreteness and rationality has been proven for $F$-jumping coefficients of principal ideals in $F$-finite regular local rings. In this paper we extend this result to excellent (but not necessarily $F$-finite) regular local rings (note that an $F$-finite ring is automatically excellent). We expect that discreteness and rationality hold for the $F$-jumping coefficients of every ideal in some large class of regular rings (perhaps all excellent regular rings), but this remains to be proven. In anticipation of this development we develop the results in Sections 2 and 3 assuming the barest minimum about $R$ (namely, condition (*) of Section 2). 

An important aspect of our paper is a novel method in the proof of our main result (see Section 6) which is substantially different from \cite{Blickle-Mustata-Smith-2}.
Unlike \cite{Blickle-Mustata-Smith-2} we do not use $D$-modules. We use Frobenius actions on the injective hull of the residue field of $R$ (see Section 5).
This method germinated in \cite[p. 10]{Katzman} in a very simple proof, without $D$-modules,
of a key result of \cite{ABL}. This method undoubtedly holds a potential for more applications.

%We then approach the general problem where the ideal $\mathfrak{a}$ in question is generated by $g_1, \dots, g_s\in R$ by
%using a further generalization of test ideals introduced by Nobuo Hara and Shunshike Takagi in \cite{Hara-Takagi} and later
%elaborated in \cite{Takagi}, namely that of the generalized test-ideals
%$\tau(g_1^{c_1} \dots g_s^{c_s})$ defined for all non-negative $c_1, \dots, c_s\in\mathbb{R}$.
%A crucial property of these ideals is the fact that
%$\tau(\mathfrak{a}^c)$ is the sum over all non-negative  $c_1, \dots, c_s\in\mathbb{R}$ adding up to $c$ of
%$\tau(g_1^{c_1} \dots g_s^{c_s})$.
%Section \ref{Section: Generalized test ideals and their properties} establishes some basic properties of these ideals
%and Section \ref{Section: Rationality and discreteness of jumping coefficients-- the general case} applies these properties to prove the main result of this paper, %Theorem \ref{Theorem: the main theorem}.

%%%%%%%%%%%%%%%%%%%%%%%%%%%%%%%%%%%%%%%%%%%%%%%%%%%%%%%%%%
\section{Preliminaries: generalized test ideals and jumping coefficients}
\label{Test ideals}
In this section $R$ is a commutative Noetherian regular ring containing a field of characteristic $p>0$. If $\mathfrak a$ is an ideal of $R$, we denote $\mathfrak a^{[p^e]}$ the ideal generated by the $p^e$-th powers of the elements of $\mathfrak a$.

\begin{defi} If $\mathfrak a$ is an ideal of $R$, we set the ideal $I_e(\mathfrak a)$ to be the intersection of all the ideals $I$ of $R$ such that $I^{[p^e]}\supset \mathfrak a$.
\end{defi}

In this paper we deal with rings $R$ satisfying the following condition:

\bigskip

\noindent (*) $\ \ \ \ \ \ \ \ \ \ \ \ \ \ \ \ \ \ \ \ \ \ \ \ R$ is regular and $I_e(\mathfrak a)^{[p^e]}\supset \mathfrak a$ for all $\mathfrak a$ and all $e$.

\bigskip

Thus $R$ satisfies (*) if and only if for all $\mathfrak a$ and $e$ the set of the ideals $I$ such that $I^{[p^e]}\supset \mathfrak a$ has a unique minimal element, denoted $I_e(\mathfrak a)$ (in  \cite{Blickle-Mustata-Smith-1} $I_e(\mathfrak a)$ is denoted $\mathfrak a^{[1/p^e]}$).

In \cite[2.3]{Blickle-Mustata-Smith-1} it is shown that an $F$-finite regular (not necessarily local) ring satisfies (*). We will show in Section 4 that an excellent regular local ring also satisfies (*). 

In this section we list some basic properties of rings satisfying (*). All of these properties were proven in 
 \cite{Blickle-Mustata-Smith-1} for $F$-finite regular rings. 

If $R$ satisfies (*), then $(\mathfrak a^{\lceil cp^e\rceil})^{[1/p^e]}\subset (\mathfrak a^{\lceil cp^{e+1}\rceil})^{[1/p^{e+1}]}$, where $\lceil x\rceil$ is the smallest integer greater than or equal to $x$. This is proven in \cite[2.8]{Blickle-Mustata-Smith-1} for an $F$-finite regular ring $R$, but the proof uses only property (*). Since $R$ is Noetherian, the ascending chain 
$$(\mathfrak a^{\lceil cp\rceil})^{[1/p]}\subset (\mathfrak a^{\lceil cp^{2}\rceil})^{[1/p^{2}]}\subset (\mathfrak a^{\lceil cp^3\rceil})^{[1/p^3]}\subset \cdots$$
stabilizes. Following \cite[Definition 2.9]{Blickle-Mustata-Smith-1}, we define the generalized test ideal of $\mathfrak a$ with coefficient $c$, denoted $\tau(\mathfrak a^c)$, as the stable value of this chain, i.e. $\tau(\mathfrak a^c)=(\mathfrak a^{\lceil cp^e\rceil})^{[1/p^e]}$ for all sufficiently large $e$.

We quote the following two results from \cite{Blickle-Mustata-Smith-1}, whose proofs use only condition (*), rather than the stronger assumption of $F$-finiteness.

\begin{prop}\label{2.14}
\cite[2.14]{Blickle-Mustata-Smith-1} If $R$ satisfies condition (*), $\mathfrak a$ is an ideal of $R$ and $c$ is a non-negative real number, then there exists $\epsilon>0$ such that $\tau(\mathfrak a^c)=I_e(\mathfrak a^r)$ whenever $c<\frac{r}{p^e}<c+\epsilon$.
\end{prop}

\begin{cor}\label{2.16}
\cite[2.16]{Blickle-Mustata-Smith-1} If $R$ satisfies condition (*), $\mathfrak a$ is an ideal of $R$ and $c$ is a non-negative real number, then there exists $\epsilon>0$ such that $\tau(\mathfrak a^c)=\tau(\mathfrak a^{c'})$ for every $c'\in [c, c+\epsilon)$.
\end{cor}

Following \cite[Definition 2.17]{Blickle-Mustata-Smith-1}, we define a positive real number $c$ to be an $F$-jumping coefficient for an ideal $\mathfrak a$ if $\tau(\mathfrak a^c)\ne\tau(\mathfrak a^{c-\epsilon})$ (i.e. the containment $\tau(\mathfrak a^c)\subset\tau(\mathfrak a^{c-\epsilon})$ is strict) for every $\epsilon>0$.

\begin{lem}\label{closed}
If $R$ satisfies condition (*) and $\mathfrak a$ is an ideal of $R$, then the set of the $F$-jumping coefficients of $\mathfrak a$ is closed in $\mathbb R$.
\end{lem}

\begin{proof} If $\lambda$ is an accumulation point of the $F$-jumping coefficients of $\mathfrak a$, then Corollary \ref{2.16} implies that there must exist an increasing sequence of $F$-jumping coefficients $\lambda_1<\lambda_2<\lambda_3<\cdots$ converging to $\lambda$. Since each $\lambda_i$ is an $F$-jumping coefficient, each containment $\tau(\mathfrak a^{\lambda_{n+1}})\subset\tau(\mathfrak a^{\lambda_{n}})$ is strict, hence each containment $\tau(\mathfrak a^{\lambda})\subset \tau(\mathfrak a^{\lambda_n})$ also is strict which shows that $\lambda$ is an $F$-jumping coefficient of $\mathfrak a$.
\end{proof}

Next we quote another result from \cite{Blickle-Mustata-Smith-1}, whose proof also uses only condition (*), rather than the stronger condition of $F$-finiteness.

\begin{prop}\label{3.4(2)}
\cite[3.4(2)]{Blickle-Mustata-Smith-1} If $R$ satisfies condition (*), $\mathfrak a$ is an ideal of $R$ generated by $m$ elements, and $\alpha>m$ is an $F$-jumping coefficient for $\mathfrak a$, then $\alpha - 1$ also is an $F$-jumping coefficient for $\mathfrak a$.
\end{prop}

And finally, we quote a result from \cite{Blickle-Mustata-Smith-1} that requires a new proof.

\begin{prop}\label{Proposition: pc is a jumping coefficient}
\cite[3.4(1)]{Blickle-Mustata-Smith-1} Assume $R$ satisfies condition (*) and let $\mathfrak a$ be an ideal of $R$. If $c$ is an $F$-jumping coefficient of $\mathfrak{a}$, so is $p c$.
\end{prop}
\begin{proof}
Assume that $c$ is an $F$-jumping coefficient of $\mathfrak{a}$, i.e.
$\tau(\mathfrak{a}^c) \subsetneq \tau(\mathfrak{a}^{c-\epsilon})$ for all $\epsilon>0$.
Choose an integer $e$ so large that both
$\tau( \mathfrak{a}^c )=I_{e+1} \left( \mathfrak{a}^{\lceil c p^{e+1} \rceil}\right)$ and
$\tau( \mathfrak{a}^{pc} )=I_e \left( \mathfrak{a}^{\lceil c p^{e+1} \rceil}\right)$.
Since
$$\mathfrak{a}^{\lceil c p^{e+1} \rceil}\subset
I_{e+1} \left( \mathfrak{a}^{\lceil c p^{e+1} \rceil}\right)^{[p^{e+1}]}=
\tau\left( \mathfrak{a}^c \right)^{[p^{e+1}]}=
\left(\left(\tau \left(\mathfrak{a}^c\right)\right)^{[p]} \right)^{[p^{e}]},$$
we see that
$$\tau(\mathfrak{a}^{pc})=
I_e \left( \mathfrak{a}^{\lceil c p^{e+1} \rceil}\right)\subset
\left(\tau \left(\mathfrak{a}^c\right) \right)^{[p]} $$
and so
$$I_1\left( \tau(\mathfrak{a}^{pc}) \right)\subset \tau \left(\mathfrak{a}^c\right) .$$
In fact $$I_1 \left(\tau(\mathfrak{a}^{pc})  \right) = \tau \left(\mathfrak{a}^c\right),$$
otherwise choose an ideal $\mathfrak{b}\subset R$ strictly smaller than
$\tau \left(\mathfrak{a}^c\right) $ for which
$\tau(\mathfrak{a}^{pc})\subset \mathfrak{b}^{[p]}$. But then for all large $e$ we must have
$$\mathfrak{a}^{\lceil c p^{e+1} \rceil} \subset \tau( \mathfrak{a}^{pc} )^{[p^e]} \subset
\mathfrak{b}^{[p^{e+1}]}$$
hence
$$\tau( \mathfrak{a}^c )=
I_{e+1} \left( \mathfrak{a}^{\lceil c p^{e+1} \rceil}\right) \subset
I_{e+1} \left( \mathfrak{b}^{[p^{e+1}]} \right)= \mathfrak{b} \subsetneq \tau \left(\mathfrak{a}^c\right), $$
a contradiction.
Similarly we can show that
$$I_1 \left( \tau(\mathfrak{a}^{p(c-\epsilon/p)})  \right) = \tau \left(\mathfrak{a}^{c-\epsilon}\right) $$ for every $\epsilon>0$.
Now, if $pc$ is not an $F$-jumping coefficient, then for some $\epsilon>0$ we have
$$\tau\left(\mathfrak{a}^c\right) = I_1\left( \tau(\mathfrak{a}^{pc})  \right)= I_1 \left(\tau(\mathfrak{a}^{p(c-\epsilon/p)})\right)=
\tau\left(\mathfrak{a}^{c-\epsilon/p}\right)$$
contradicting the fact that $c$ is an $F$-jumping coefficient.
\end{proof}

\section{A necessary and sufficient condition for the rationality and discreteness of $F$-jumping coefficients}
\label{Section: RD}
%%%%%%%%%%%%%%%%%%%%%%%%%%%%%%%%%%%%%%%%%%%%%%%%%%%%%%%%%%%%%%%%%%%%%%%%%%%%%%%%%%%%%%%%%
The main result of this section is the following

\begin{thm}\label{mainRD}
Assume $R$ satisfies condition (*). If the set of the $F$-jumping coefficients of $\mathfrak a$ has no rational accumulation points, then the set of the $F$-jumping coefficients of $\mathfrak a$ is discrete and every $F$-jumping coefficient of $\mathfrak a$ is rational.
\end{thm}

Assume $\mathfrak a$ is generated by $m$ elements. We denote the modulo $m$ part of a real number $s$ with $\{ s \}$, that is $\{s\}=s-m[\frac{s}{m}]\in [0,m)$, where $[x]$ denotes the integer part of $x$.

\begin{lem}\label{inf}
If $s$ is an irrational number, the set $\{ \{p^e s\} \,|\, e\geq 1 \}$ is infinite.
\end{lem}
\begin{proof}
Let $0.s_1 s_2  \dots $ be the fractional part of the base $p$ expansion of $s$.
Then the fractional part of $\{p^e s\}$ is $0. s_{e+1} s_{e+2} \dots$. If
$\{p^e s\}=\{p^j s\}$ for some $e<j$, then
the base $p$ expansion of $s$ is eventually periodic with period $s_{e+1} s_{e+2} \dots s_{j}$ of length
$j-e$ starting at the $e+1$th digit.
\end{proof}

\begin{prop}\label{Proposition: properties of $C$}
Let $C\subset [0,m]\setminus\mathbb{Q}$ be a set with the following properties:
\begin{enumerate}
\item [(a)] for all $c\in C$ there exists an $\epsilon>0$ such that $(c,c+\epsilon) \cap C=\emptyset$
\item [(b)] $C$ is closed in $[0,m]$,
\item [(c)] $\{p^e c\}\in C$ for all $c\in C$ and all $e\geq 1$.
\end{enumerate}
Then $C$ is empty.
\end{prop}
\begin{proof}
Assume that $C\neq \emptyset$ and let $\mathcal{C}$ be the collection of non-empty subsets of $C$ which satisfy the three properties above.
Endow $\mathcal{C}$ with a partial order $\preceq$ defined by $A\preceq B \Leftrightarrow A\supseteq B$ for all $A,B\in \mathcal{C}$.

Let $\mathcal{D}$ be a chain in $(\mathcal{C}\preceq)$; we now show that
$\displaystyle C_0=\cap_{D\in \mathcal{D}} D \in \mathcal{C}$. Clearly, $C_0$ satisfies properties (a), (b) and (c) above, so we just need to show that
$C_0\neq \emptyset$.

For each $D\in \mathcal{D}$ let $x_D=\inf(D)\in D$ and let $x=\sup \{ x_D \,|\, D\in \mathcal{D} \}$; we show that $x\in C_0$.
Assume that $x\notin D^\prime$ for some $D^\prime\in \mathcal{D}$. Since $D^\prime$ is closed, there exists an $\epsilon>0$ such that
$(x-\epsilon,x+\epsilon)\cap D^\prime =\emptyset$;
also, since $x_{D^\prime}\leq x$, we deduce that $x_{D^\prime}\leq x-\epsilon$.
Now pick any $D\in \mathcal{D}$; either
$D\subset D^\prime$, in which case $(x-\epsilon,x+\epsilon)\cap D =\emptyset$ or
$D\supseteq D^\prime$, in which case $x_D\leq x_{D^\prime}\leq x-\epsilon$ and in either case
$x_D\notin (x-\epsilon, x]$ for all $D\in \mathcal{D}$, which is impossible.

Apply now Zorn's Lemma to obtain a maximal element $M\in \mathcal{C}$ with respect to $\preceq$, i.e.,
a minimal element in $\mathcal{C}$ with respect to inclusion.
Let $M^\prime$ be the set of accumulation points of $M$.
Clearly $M^\prime$ satisfies property (a) above by virtue of being a subset of $M$, it satisfies (b) because it is a set of accumulation points
and it satisfies (c)
because the  functions $ x \mapsto \{ p^e x \}$ are continuous on the set of the irrational points of $[0,m]$, hence they map the set of accumulation points of $M$ to itself.

Since $M$ is closed there exists a minimal $m\in M$, and property (a) implies that
$m$ is an isolated point of $M$. Now $m\in M\setminus M^\prime$ and the minimality of $M$ implies that $M^\prime\notin\mathcal{C}$, hence
$M^\prime=\emptyset$. We deduce now that $M$ must be finite, and since it satisfies property (c), the previous lemma shows that it contains no irrational points.
We conclude that $M$ is empty, contradicting the fact that $M\in\mathcal{C}$.
\end{proof}

{\it Proof of Theorem \ref{mainRD}.}
Let $C$ denote the set of the accumulation points of the $F$-jumping coefficients of $\mathfrak a$ in the interval $[0,m]$.
This set satisfies the three conditions listed in Proposition \ref{Proposition: properties of $C$}:
(a) holds because of Corollary \ref{2.16}, (b) holds because the set of the $F$-jumping coefficients of $\mathfrak a$ is closed (see Lemma \ref{closed}) and the set of the accumulation points of a closed set is closed and
(c) follows from Propositions \ref{3.4(2)} and \ref{Proposition: pc is a jumping coefficient}. By assumption, all points of $C$ are irrational. Hence Proposition \ref{Proposition: properties of $C$} implies that $C$ is empty, i.e. the set of the $F$-jumping coefficients of $\mathfrak a$ is discrete. 

If there were an irrational $F$-jumping coefficient $s$, then each $\{p^es\}$ also would be an $F$-jumping coefficient by Propositions \ref{3.4(2)} and \ref{Proposition: pc is a jumping coefficient}. Hence by Lemma \ref{inf} there would be infinitely many $F$-jumping coefficients in the interval $[0,m]$ and therefore the set of the accumulation points would be non-empty, contrary to what has just been shown. This means that the $F$-jumping coefficients are all rational.
\qed
%%%%%%%%%%%%%%%%%

\section{Excellent regular local rings}
\label{excellent}

It was shown in \cite{Blickle-Mustata-Smith-1} that regular $F$-finite rings satisfy condition (*), and it was shown in \cite{Katzman} that complete regular local rings satisfy condition (*). In this section we show that in the local case the condition of $F$-finiteness and the condition of completeness can be relaxed. Namely, the main result of this section is 

\begin{thm}\label{exc}
An excellent regular local ring $R$ satisfies condition (*).
\end{thm}

\emph{Proof.} First we consider the case that $R$ is complete. This case is done in \cite[5.3]{Katzman}, but the statement is somewhat different; we reproduce a proof to avoid any misunderstandings. Recall that by definition, $I_e(\mathfrak a)$ is the intersection of all the ideals $I$ such that $I^{[p^e]}\supset \mathfrak a$. A well-known theorem of Chevalley (see Theorem 1 in Chapter 5 of \cite{Northcott}) says that if an ideal $J$ is the intersection of some set $S$ of ideals in a complete local ring, then for every positive integer $t$ there exists an ideal $I\in S$ such that $I\subset J+\mathfrak m^t$. Hence for every integer $t$ there is an ideal $I$ such that $I^{[p^e]}\supset \mathfrak a$ and $I\subset I_e(\mathfrak a)+\mathfrak m^t$. Therefore $I^{[p^e]}\subset I_e(\mathfrak a)^{[p^e]}+\mathfrak (m^t)^{[p^e]}\subset I_e(\mathfrak a)^{[p^e]}+\mathfrak m^{tp^e}$ which implies that $\mathfrak a\subset I_e(\mathfrak a)^{[p^e]}+\mathfrak m^{tp^e}$ for every $t$. But the intersection of $I_e(\mathfrak a)^{[p^e]}+\mathfrak m^{tp^e}$, as $t$ runs through all positive integers, is $I_e(\mathfrak a)^{[p^e]}$, hence $\mathfrak a\subset I_e(\mathfrak a)^{[p^e]}$. This completes the case where $R$ is complete.

For the general case, let $\hat R$ denote the completion of $R$ with respect to the maximal ideal. It has just been proven that $I_e(\mathfrak a\hat{R})^{[p^e]}\supset \mathfrak a\hat R$. Let $H'=I_e(\mathfrak a\hat{R})$ and let $H=H'\cap R$. Since $R$ is excellent, \cite[6.6]{LySm} implies $(H')^{[p^e]}\cap R=H^{[p^e]}$, hence $\mathfrak a \subset H^{[p^e]}$. If $L$ is any ideal of $R$ such that $\mathfrak a \subset L^{[p^e]}$, then $\mathfrak a\hat{R}\subset (L\hat{R})^{[p^e]}$. The minimality of $H'$ tells us that $L\hat{R}\supset H'$. But then $L=L\hat{R}\cap R\supset H'\cap R=H$. Therefore, $I_e(\mathfrak a)$ is the unique minimal element of the set of the ideals $I$ of $R$ such that $I^{[p^e]}\supset \mathfrak a$. Thus $R$ satisfies property (*). \qed

\smallskip

It stands to reason that every excellent regular (not necessarily local) ring satisfies condition (*) but this remains to be proven.

Theorem \ref{exc} shows that all the results of the preceding two sections are valid for excellent regular local rings.

\section{Frobenius actions on Artinian modules}
\label{S:Frob}
%%%%%%%%%%%%%%%%%%%%%%%%%%
In this section we introduce the tools that will be used in the next section to prove our main results. In this section and the next $R$ is an excellent regular local ring. We begin with a generalization of a well-known result.
\begin{lem}\label{t-fold}
Let $N$ be an Artinian $R$-module and let $\phi:N\to N$ be an action of the $t$-fold Frobenius, i.e. a morphism of abelian groups such that
$\phi(rm)=r^{p^t}\phi(m)$ for all $r\in R$ and $m\in N$.
Let $M\subset N$ be the submodule of the nilpotent elements with respect to $\phi$, i.e. $M=\{m\in N|\phi^s(m)=0 {\rm\ for\ some\ } s\}$.
Then there is $s$ such that $\phi^s(m)=0$ for all $m\in M$.
\end{lem}

\begin{proof} Clearly, $\phi$ acts on $M$ in a natural way, i.e. $\phi(m)\in M$ for all $m\in M$. Hence for $t=1$ the result is well-known \cite[1.11]{HaSp}, \cite[4.4]{Lyubeznik}. For $t>1$ a proof may be obtained, for example, along the lines of \cite[Remark 5.6a]{Lyubeznik}, i.e. constructing a theory of $F^t$-modules and applying it to prove the lemma via an analog of \cite[4.2]{Lyubeznik} for $F^t$-modules in the same way as \cite[4.4]{Lyubeznik} was deduced from \cite[4.2]{Lyubeznik} in \cite{Lyubeznik}.
\end{proof}

The Frobenius map $R\stackrel{r\mapsto r^p}{\longrightarrow}R$ induces a natural action $f:H^{{\rm dim}R}_{\mathfrak m}(R)\to H^{{\rm dim}R}_{\mathfrak m}(R)$ on the top local cohomology module of $R$ with support in the maximal ideal. It is a map of abelian groups such that $f(rh)=r^pf(h)$ for every $r\in R$ and every $h\in H^{{\rm dim}R}_{\mathfrak m}(R)$.

\begin{lem}
Let $R$ be an excellent regular local ring, let $H=H^{{\rm dim}R}_{\mathfrak m}(R)$ and let $f:H\to H$ be the natural action as above. Let $L\subset H$ be an $R$-submodule and let $I\subset R$ be the annihilator ideal of $L$. We denote $L^{[p^e]}$ the $R$-submodule of $H$ generated by the elements $f^e(x)$ as $x$ runs through all the elements of $L$. The annihilator ideal of $L^{[p^e]}$ in $R$ is $I^{[p^e]}$.
\end{lem}

\begin{proof}
Since $H$ is Artinian, $H, L$ and $L^{[p^e]}$ are all $\hat R$-modules in a natural way, where  $\hat R$ is the completion of $R$ with respect to the maximal ideal. Assume the lemma holds for $\hat R$. Let $\tilde I$ be the annihilator of $L$ in $\hat R$. Then the annihilator of  $L^{[p^e]}$ in $\hat R$ (resp. in $R$) is $\tilde I^{[p^e]}$ (resp. $\tilde I^{[p^e]}\cap R$). Since $I=\tilde I\cap R$, it follows from \cite[6.6]{LySm} that $\tilde I^{[p^e]}\cap R=I^{[p^e]}$. Thus it is enough to prove the lemma assuming that $R$ is complete, and we so assume. 

Let $R^{(e)}$ be the $R$-bimodule whose underlying abelian group is $R$, the left $R$-module structure is the usual one and the right $R$-module structure is given by $r'\cdot r=r^{p^e}r'$. Since $R$ is regular, $R^{(e)}$ is a flat right $R$-module, hence $R^{(e)}\otimes_RL$ is a submodule of  $R^{(e)}\otimes_RH$. The $R$-linear map $\phi:R^{(e)}\otimes_RH\to H$ that sends $r'\otimes h$ to $r'f^e(h)$ is well-known to be an isomorphism and $L^{[p^e]}=\phi(R^{(e)}\otimes_RL)$. Hence the annihilator of $L^{[p^e]}$ is the same as the annihilator of $R^{(e)}\otimes_RL$.

For an Artinian module $N$ the annihilator of $N$ in $R$ is the same as the annihilator of $D(N)$, where $D(-)={\rm Hom}_R(-,E)$ is the Matlis duality functor. Hence the annihilators of $L$ and $R^{(e)}\otimes_RL$ are the same as, respectively, the annihilators of $D(L)$ and $D(R^{(e)}\otimes_RL)$. Since $R$ is regular, $H\cong E$ and, since $R$ is complete, $D(L)\cong R/I$. Since $L$ is Artinian, \cite[4.1]{Lyubeznik} implies that $D(R^{(e)}\otimes_RL)\cong R^{(e)}\otimes_RD(L)$, i.e. $D(R^{(e)}\otimes_RL)\cong R^{(e)}\otimes_RR/I\cong R/I^{[p^e]}$ (note that the functor $R^{(e)}\otimes_R(-))$ is just the Frobenius functor denoted $F(-)$ in \cite{Lyubeznik}).
\end{proof}

%%%%%%%%%%%%%%%%%%%%%%%%%
\section{Main results}
\label{MR}

The main result of this section and the whole paper is Theorem \ref{main}. In this section $R$ is excellent, regular and local, and hence satisfies condition (*) by Theorem \ref{exc}. 

Since $R$ is regular, $E$ is (non-canonically) isomorphic to $H^{{\rm dim}R}_{\mathfrak m} (R)$. The canonical Frobenius action $f:H^{{\rm dim}R}_{\mathfrak m} (R)\to H^{{\rm dim}R}_{\mathfrak m} (R)$ gives rise to a non-canonical Frobenius action $f:E\to E$ via an isomorphism $E\cong H^{{\rm dim}R}_{\mathfrak m} (R)$. We fix one such Frobenius action $f$ on $E$ once and for all. 

Given a fixed $g\in R$ we can define for all $a\in  \mathbb{N}$ and $\beta\in  \mathbb{N}$
an $R[\Theta_{a,\beta}; f^\beta]$-left-module structure
(cf.~\cite[Section 2]{Katzman} for notation and properties of these skew-polynomial rings) on $E$
given by $\Theta_{a,\beta} m=g^a f^\beta (m)$ for every $m\in E$.
For all integers $e$ and $q$ we define
$$\psi_e(q)=1+q+\dots+q^{e-1}=\frac{q^e-1}{q-1} .$$

\begin{thm}\label{Theorem: Nilpotents}[cf.~Theorem 4.6 in \cite{Katzman}]
With $\Theta$ as above, set $N_s=\{ m\in E \,|\, \Theta^s m=0 \}$.
We have
$$N_s=\Ann_E I_{s \beta}(g^{ a\psi_s(p^\beta) }) .$$
\end{thm}
\begin{proof}
Using the relation $f^{\beta}(g^a(-))=g^{ap^{\beta}}f^{\beta}(-)$ one proves by induction on $s$ that $\Theta^s m=g^{ a\psi_s(p^\beta) }f^{s\beta}(m)$ for every $m\in E$. This implies that an $R$-submodule $L$ of $E$ is a submodule of $N_s$ if and only if $g^{ a\psi_s(p^\beta) }$ annihilates $L^{[p^{s\beta}]}$, i.e. if and only if $g^{ a\psi_s(p^\beta) }$ belongs to the annihilator ideal of $L^{[p^{s\beta}]}$. If $I\subset R$ is the annihilator ideal of $L$, the preceding lemma implies that the annihilator ideal of $L^{[p^{s\beta}]}$ is $I^{[p^{s\beta}]}$. But the smallest ideal $I$ of $R$ such that $g^{ a\psi_s(p^\beta) }\in I^{[p^{s\beta}]}$ is $I_{s \beta}(g^{ a\psi_s(p^\beta) })$.
\end{proof}

Since $N_1\subset N_2\subset N_3\subset\dots$ form an ascending chain, their annihilator ideals form
a descending chain, i.e. $I_{\beta}(g^{ a\psi(p^\beta) })\supset I_{2 \beta}(g^{ a\psi_2(p^\beta) })\supset\dots$.

\begin{cor}\label{chain}
The descending chain of ideals $I_{\beta}(g^{ a\psi(p^\beta) })\supset I_{2 \beta}(g^{ a\psi_2(p^\beta) })\supset\dots$ stabilizes.
\end{cor}

\begin{proof} By Lemma \ref{t-fold}, the ascending chain of submodules $N_1\subset N_2\subset\dots$ stabilizes, hence the descending chain of their annihilators stabilizes as well.
\end{proof}

In particular, for $\beta=1$ and $a=p-1$ we recover \cite[4.2]{ABL} for a complete regular
local ring without assuming that $R$ is $F$-finite and without $D$-modules in the proof (cf. \cite[p. 10]{Katzman}).

To simplify notation for generalized test ideals of
principal ideals we write $\tau(g^c)$ instead of $\tau\big( (gR)^c \big)$.

\begin{prop}\label{Theorem: no special accumulation points}
Let  $a,\beta>0$ be integers and write $\gamma=a/(p^\beta-1)$.
There exists a $c\in(0,\gamma)$ for which $\tau(g^d)=\tau(g^c)$ for all $d\in(c,\gamma)$.
\end{prop}
\begin{proof}
For all integers $r>0$ and $e\geq 1$, the generalized test ideal $\tau(g^{r/p^e})$ is just $I_e(g^r)$ because the ideal $(g^r)$ is principal.
Now for all integers $\beta,e\geq 1$,
\begin{equation}\label{eqn1}
I_{s\beta}(g^{a(1+p^\beta+\dots + p^{\beta(s-1)})})=
\tau\left( g^\frac{a(1+p^\beta+\dots + p^{\beta(s-1)})}{p^{s\beta}} \right)
\end{equation}
and
$$\frac{ a(1+p^\beta+\dots + p^{\beta(s-1)})} {p^{s\beta}}=
\frac{a}{p^{s\beta}} \frac{p^{s\beta}-1}{p^\beta-1}$$
is an increasing sequence which converges to $\gamma$ as $s\rightarrow \infty$.
By Corollary \ref{chain} the left hand side of (\ref{eqn1}) stabilizes (say for $s\geq \nu$). Hence so must the right hand side, i.e. we may take
$c=\frac{a(1+p^\beta+\dots + p^{\beta(\nu-1)})}{p^{\nu\beta}}$.
\end{proof}

\begin{cor}
The set of the $F$-jumping coefficients of $g$ cannot have an accumulation point of the form
$\frac{a}{p^{d}(p^{\beta}-1)}$ with $a, d, \beta\in \mathbb{N}$.
\end{cor}
\begin{proof}[Proof]
Otherwise, there would be a sequence of $F$-jumping coefficients $\{c_n\}_{n\geq 1}$
converging to $\frac{a}{p^{d}(p^{\beta}-1)}$. Now Propositions \ref {3.4(2)} and \ref{Proposition: pc is a jumping coefficient} imply that
$\{p^{d}c_n\}$ is a sequence of $F$-jumping coefficients converging to
$\frac{a}{p^{\beta}-1}$, a contradiction.
\end{proof}

And finally, our main result.

\begin{thm}\label{main}
Let $R$ be an excellent regular local ring and let $g\in R$ be an element. The set of $F$-jumping coefficients of $g$ is discrete and every $F$-jumping coefficient of $g$ is rational.
\end{thm}
\begin{proof}[Proof]
According to Theorem \ref{mainRD} it is enough to show that the set of the $F$-jumping coefficients of $g$ has no rational accumulation points. Let $\frac{m}{n}$ be a positive rational number. Write $n=p^{\alpha}n_1$, where $\alpha\geq 0$ and $p$ does not divide $n_1$. By the preceding corollary it suffices to prove that there exists
$\beta\in\mathbb{N}$ such that $n_1|(p^{\beta}-1)$. Since $(p,n_1)=1$, Euler's Theorem enables one to set $\beta=\varphi(n_1)$, where $\varphi$ is Euler's function.
\end{proof}

%%%%%%%%%%%%%%%%%%%%%%%%%%%%%%%%%%%%%%%%%%%%%%%%%%%%%%%%%%%%%%%%%%%%%%%%%%%%%%%%%%%%%%%%%

\section{More on Frobenius actions}
In this section we have included some results that are not used in the proof of our main results but are of independent interest.
Given an $R[\Theta_{a,\beta}; f^\beta]$-module structure on $E$, we investigate the nilpotent $R[\Theta_{a,\beta}; f^\beta]$-submodules of $E$ defined as
$$\Nil_{a,\beta}=\left\{ m\in E \,|\, \Theta_{a,\beta}^e m=0 \text{ for some } e>0 \right\} .$$
Denote the set of these nilpotent submodules $\left\{ \Nil_{a, \beta} \,|\, a\geq 0, \beta>0 \right\}$ with $\mathcal{N}$. For an integer $r\geq 0$ we write $I_{r,e}$ for $I_e(g^r)$. 

\begin{prop}\label{Theorem: bijection}
The map $\tau(g^c) \mapsto \Ann_E \tau(g^c)$ defined for all $c\geq 0$
is a bijection between the set of generalized test-ideals of $g$ and $\mathcal{N}$.
\end{prop}
\begin{proof}
Pick any $c\geq 0$; we first show that $\Ann_E \tau(g^c)\in \mathcal{N}$.
Use Proposition \ref{2.14} to pick an $\epsilon>0$ such that
$\tau(g^c)=I_{r,e}$ for all positive integers $r,e$ which satisfy $c<r/p^e<c+\epsilon$.
Use the fact that the set
$$\left\{ \frac{a}{p^\beta -1} \,|\, a, \beta\in \mathbb{Z}, a\geq 0, \beta>0 \right\}$$
is dense in the interval $(c,c+\epsilon)$ to pick  $a\in  \mathbb{N}_0$ and  $\beta \in  \mathbb{N}$ such that
$c< a/(p^\beta-1) < c+\epsilon$.
The increasing sequence $a \psi_e(p^\beta)/p^{\beta e}$ converges to $a/(p^\beta-1)$ as $e\rightarrow \infty$;
we may pick an $e_0\gg 1$ such that $c< a \psi_{e_0}(p^{\beta e_0})/p^{e_0} < a/(p^\beta-1) < c+\epsilon$, and we deduce that
$\tau(g^c)=I_{a \psi_{e}(p^\beta), \beta e}$ for all $e\geq e_0$.
It follows from Lemma \ref{t-fold} and Theorem \ref {Theorem: Nilpotents} that we may also pick an $e_1\geq e_0$ such that
$$\Nil_{a,\beta}=\Ann_E I_{a \psi_{e}(p^\beta), \beta e}$$ for all $e\geq e_1$.
Now
$$\Ann_E\tau(g^c)=\Ann_E I_{a \psi_{e_1}(p^\beta), \beta e_1}=\Nil_{a,\beta} $$
and Matlis duality implies that the function $\tau(g^c) \mapsto \Ann \tau(g^c)$ is injective.

Pick now any $\Nil_{a,\beta}\in\mathcal{N}$ and pick an $e_0\gg 1$ such that
$\Nil_{a,\beta} =\Ann_E I_{a \psi_{e_0}(p^\beta), \beta e_0}$.
Let $\displaystyle c=\frac{a \psi_{e_0}(p^\beta)}{p^{\beta e_0}}$
and pick $e_1\gg \beta e_0$ such that
$\displaystyle \tau(g^c)=I_{\lceil c p^{e_1}\rceil, e_1}$.
Now
$$ \tau(g^c)=I_{\lceil \frac{a \psi_{e_0}(p^\beta)}{p^{\beta e_0}} p^{e_1}\rceil, e_1} =
I_{a \psi_{e_0}(p^\beta)p^{e_1 - \beta e_0} , e_1} =
I_{a \psi_{e_0}(p^\beta) , \beta e_0} $$
and
$$\Nil_{a, \beta}=\Ann_E I_{a \psi_{e_0}(p^\beta) , \beta e_0}= \Ann_E \tau(g^c) .$$
\end{proof}

Notice that the argument above shows that there exists an $\epsilon>0$
such that for all
$a\in  \mathbb{N}_0$ and $\beta\in  \mathbb{N}$ with
$c< a/(p^\beta-1) < c+\epsilon$,
$\Nil_{a,\beta}$ is constant (and equal to $\tau(g^c)$.)

\begin{cor}
The set $\mathcal{N}$ is totally ordered with respect to inclusion and
$$\displaystyle \frac{a_1}{p^{\beta_1}-1}\geq \frac{a_2}{p^{\beta_2}-1}  \Leftrightarrow \Nil_{a_1, \beta_1} \subset \Nil_{a_2, \beta_2} .$$
\end{cor}

\end{document}